\documentclass[10pt,a4paper]{article}
\usepackage[latin1]{inputenc}
\usepackage[T1]{fontenc}
\usepackage{amssymb,amsthm,amsmath,amsfonts,mathrsfs}
\usepackage{hyperref}
\usepackage{makeidx}
\usepackage{graphicx}
\usepackage[english]{babel}
\usepackage{color}
\usepackage[width=16.00cm, height=24.00cm]{geometry}
\usepackage[english]{babel}
\newtheorem{prop}{Proposition}[section]
\newtheorem{lem}{Lemma}[section]

\newtheorem{thm}{Theorem}[section]
\newtheorem{remark}{Remark}[section]

\numberwithin{equation}{section}

\newtheorem{rmq}{\textbf{Remark}}[section]

\newcommand{\thesectionwords}{\ifcase \thechapter \fi}

\usepackage{fancyhdr}
\newcommand{\be}{\begin{equation} \label}
\newcommand{\ee}{\end{equation}}
\newcommand{\R}{\mathbb{R}}
\newcommand{\N}{\mathbb{N}}
\newcommand{\eps}{\varepsilon}
\author{Loth}
\begin{document}
	\begin{center}
		\section*{  Asymptotic 
		blow-up behavior for the semilinear heat equation with   non scale invariant nonlinearity}
		$ $
	\end{center}
	
	\begin{center}
		  Loth Damagui CHABI\\
		$ $
		
	\end{center}
\begin{abstract}
 We characterize the asymptotic behavior near blowup points for positive solutions of the semilinear heat equation 
 \begin{equation*}
 \partial_t u-\Delta u =f(u),
 \end{equation*}
  for nonlinearities which are genuinely non scale invariant,
 unlike in the standard case $f(u)=u^p$.
 Indeed, our results apply to a large class of nonlinearities of the form 
	$f(u)=u^pL(u)$, where $p>1$ is Sobolev subcritical and  $L$ is a slowly varying function at infinity
	(which includes for instance logarithms and their powers and iterates, as well as some strongly oscillating functions).	
 
 More precisely, denoting by $\psi$ the unique positive solution of the corresponding ODE $y'(t)=f(y(t))$ which blows up at the same time $T$, we show that if $a\in\Omega$ is a blowup point of $u$, then  \begin{equation*}
 	\lim_{t\to T}\frac{u(a+y\sqrt{T-t},t)}{\psi(t)}= 1,\quad \text{uniformly for $y$ bounded.}
 	\end{equation*}
	 Additional blow-up properties are obtained, including the compactness of the blow-up set for the Cauchy problem 
	with decaying initial data.
\end{abstract}
{\bf Key words:} Semilinear heat equation,   
 asymptotic blowup behavior,  blow-up set, regular variation, weighted energy.
	\section{Introduction}
We consider the semilinear heat equation 
\begin{equation}
\begin{cases}
u_t-\Delta u=f(u),&x\in \Omega,\ t>0,\\
u=0,& x\in \partial\Omega,\ t>0,\\ 
u(x,0)=u_0(x),&x \in \Omega. 
\end{cases}\label{eqE2}
\end{equation}
 Throughout this article, $\Omega$ is a, possibly unbounded, uniformly smooth domain of $\mathbb{R}^n$ $(n\ge 1)$
and $f\in C^1([0,\infty))$  
satisfies $f(0)\ge 0$ and $f(s)>0$ for $s$ large.  
It is well known that, for $u_0\in L^\infty(\Omega)$, problem \eqref{eqE2} has a unique 
 nonnegative classical solution.  Throughout this article we will denote by $u$ this solution and by 
$T=T(u_0)\in(0,\infty]$ its maximal existence time.
If $f$ has superlinear growth in the sense that $1/f$ is integrable at infinity then
(see,~e.g.,~\cite[Section~17]{quittner2019superlinear}), under suitable largeness condition on the initial data, $u$ blows up in finite time, i.e $ T < \infty$ and
\begin{equation*}
\lim_{t\to T} ||u(t)||_\infty=\infty.
\end{equation*}
In this case, $T$ is called the blowup time of $u$. Given $a \in \overline{\Omega}$, we say that $a$ is a blowup point of $u$ if there exists $(a_j,t_j)\to(a,T)$ such that $|u(a_j,t_j)|\to\infty$ as $j\to \infty$.

The asymptotic behavior of blowup solutions for problem \eqref{eqE2} has been studied in great detail in the special case $f(u) = |u|^{p-1}u$,
especially in the Sobolev subcritical range $p\in (1,p_S)$ with 
$$ p_S=\begin{cases}
\frac{n+2}{n-2},&\hbox{ if $n\ge 3$} \\
\noalign{\vskip 1mm}
\infty,&\hbox{ if $n\le 2$}.
\end{cases}$$
In their fundamental work \cite{giga1985asymptotically,giga1989nondegeneracy}, Giga and Kohn have studied the local behavior of solutions near blow-up points for this range of $p$
and have discovered that, in backward self-similar parabolas, the solution behaves like the solution of the corresponding ODE, namely: \be{ell16}
\lim_{t\to T}(p-1)^{\frac{1}{p-1}}(T-t)^{\frac{1}{p-1}} u(a+y\sqrt{T-t},t)= 1. 
\ee
Later on, building on the result in \cite{giga1985asymptotically,giga1989nondegeneracy}, the sharp final blowup profiles and the corresponding refined space-time behaviors 
have been completely classified in the Sobolev subcritical range (cf.~~\cite{filippas1992refined, herrero1992blow, HV93, Vel92, Vel93b, BK94, 
	merle1998optimal, merle1998refined, souplet2019simplified}).
The complementary range $p\ge p_S$ has also been the subject 
of a number of investigations, but this range  
 exhibits more complicated
behaviors and is less understood (for instance other self-similar or non self-similar behaviors are possible; see, e.g., \cite[Section $25$]{quittner2019superlinear}  and the references therein for details).

On the other hand, the above mentioned analysis for the pure power nonlinearity in the Sobolev subcritical range heavily depends on the scale invariance 
properties of the equation, namely the fact that the equation is invariant by the transformation \begin{equation*}
u\mapsto \lambda^{\frac{2}{p-1}}u(\lambda x,\lambda^2 t),\quad  \lambda>0.
\end{equation*}
As already noted in \cite[Section $5.3$]{BB}, the precise asymptotic blow-up behavior for general nonlinearities is still a widely open problem.
The main goal of the present work and of the companion paper \cite{chso} is to partially fill this gap and to provide a precise description of the blow-up behavior
of solutions for a large class of non scale invariant nonlinearities.
The present paper is devoted to the local behavior of general solutions near arbitrary blow-up points,
whereas \cite{chso} will concentrate on radial decreasing solutions and describe the sharp final blowup profile and the refined space-time behavior
(note that \cite{chso} will make essential use of the results of the present paper).

As far as we know, the only previous study of local blow-up asymptotics for problem \eqref{eqE2} with a genuinely non-scale invariant nonlinearity \footnote{Nonlinearities with asymptotic scale invariance as $s\to\infty$
can be treated by similar methods as for the case $f(s)=s^p$, see~\cite[Section~6A]{giga1989nondegeneracy} and \cite{BK94}.} 
was recently carried out in \cite{duong2018construction} where, for the special case of the logarithmic nonlinearity $|u|^{p-1}u\log^q(2+u^2)$, 
the authors construct a special, single-point blow-up solution with a prescribed final and space-time blow-up profile.
A starting point in the approach of \cite{duong2018construction} is to rescale the problem by similarity variables and ODE renormalization, 
using as normalization factor the positive solution of the ODE $y'=f(y)$ blowing up at the same time $T$ (instead of $((p-1)(T-t))^{-1/(p-1)}$).
We will here use this idea with a different goal, and in a more systematic way as regards the nonlinearity.
Namely, for a large class of nonlinearities satisfying a suitable regular variation assumption at infinity, 
we will prove results  that can be seen as a counterpart of \cite{giga1985asymptotically,giga1989nondegeneracy},
describing the local behavior of any solution near an arbitrary blow-up point.
We will also show the compactness of the blow-up set in the case $\Omega=\R^n$  with decaying initial data.

\section{Main results}
\subsection{Statements of main results}
For $p\in\mathbb{R}$, we say that the function $f$ has regular variation at $\infty$  of index $p$ if the function $L$ defined by $L(s):=s^{-p}f(s)$ satisfies
\be{prop}
\lim_{\lambda\to\infty}\frac{L(\lambda s)}{L(\lambda)}=1\quad \text{for each }\ s>0.
\ee
A  function $L$ with the property \eqref{prop} is called a function with slow variation at $\infty$. 
 When $L$ is $C^1$ near infinity, a well-known sufficient condition for \eqref{prop} is $\lim_{s\to\infty} s\frac{L'(s)}{L(s)}=0$.
  We shall consider the following subclass of functions with regular variation, with index $p>1$:
\be{sem0}
f\in C^1([0,\infty)),\quad f(0)\ge 0, \quad\hbox{ { $f>0$} for large $s$} 
\ee
\be{sem}
{ L(s):=\frac{f(s)}{s^p}\hbox{ satisfies }  \frac{sL'(s)}{L(s)}=O\bigl(\log^{-\alpha}(s)\bigr)
\hbox{ as $s\to\infty$, for some $\alpha>\frac12$.}}
\ee
 Some examples of function $L$ with slow variation that satisfy \eqref{sem} are  given in Remark \ref{ex1} below  
 (see \cite{seneta} for a general reference  on regularly varying functions and, e.g,~\cite{bingham1989regular} and 
\cite{cirstea, souplet2022universal}). Throughout this paper, we shall denote  by $\psi(\cdot)$ the unique positive increasing solution of $y'=f(y)$ which blows up at $T<\infty$ (see at the end of this section for details). We also denote $\beta=\frac{1}{p-1}$ and $\kappa=\beta^\beta$.

\goodbreak 

\begin{thm}\label{RDT}
	Let $1<p< p_S$
	and $u_0\in L^\infty(\Omega)$ with $u_0\ge 0$.  
	 Assume \eqref{sem0}-\eqref{sem} and $T:=T_{\max} (u_0)<\infty$.
	If $a\in \Omega$ is a blow-up point of $u$, then  \be{jojo0}
	\lim_{t\to T}\frac{u(a+y\sqrt{T-t},t)}{\psi(t)}=  1,
	\ee
	uniformly on compact sets $|y|\le C$.
\end{thm}

 \begin{rmq}\label{ex1}
 	 \begin{itemize}
 		\item The conclusion of Theorem \ref{RDT}, with limit $\pm 1$, remains valid for sign changing solutions (i.e., without the assumption $u_0\ge 0$  and, for instance, extending $f$ as { an} odd function for $s<0$) provided the solution is of type I, namely:
		\be{dam1}
 		\|u(t)\|_\infty\le M\psi(t),\quad  T-\delta<t<T,
 		\ee
 		for some $M, \delta>0$.  See Theorem~\ref{progen} below. 
		
		 		\item For nonnegative initial data, 
				property \eqref{dam1} 
		follows from the recent result \cite[Theorem~3.1]{souplet2022universal},
		whenever $p\in(1,p_S)$ and $L$ has slow variation at $\infty$, hence in particular under the assumptions of Theorem \ref{RDT}
		(see Theorem~\ref{A1709241} in appendix below). 
		However, for sign-changing solutions and $p\in(1,p_S)$, property \eqref{dam1} so far is known only for $f(u)=|u|^{p-1}u$ 
		\cite{giga1985asymptotically,giga1989nondegeneracy} or for $f(u)=|u|^{p-1}u \log^q(2+u^2)$ \cite{hamza2022blow}.  
		The result also remains true
		for $p=p_S$ under assumption \eqref{dam1}
 		 (but \eqref{dam1}  is not true in general when $p=p_S$ (see \cite[Section $25$]{quittner2019superlinear}).

 	\item 	As examples of nonlinearities such that  assumptions \eqref{sem0}-\eqref{sem} are satisfied,
 	so that Theorem~\ref{RDT}  applies, we 
 	have $f(s)=s^pL(s)$ with $L$ given by:
 	$$
 	\left\{\begin{aligned}
 		&\hbox{\ $\star$\  $\log^a (K+s)$ for $K>1$ and $a\in\R$,} \\
 		\noalign{\vskip 1mm}
 		&\hbox{\ $\star$\  the iterated logarithms {$\log_m(K+s)$,\footnotemark}}  \\
 		\noalign{\vskip 1mm}
 		&\hbox{\ $\star$\  $\exp(|\log s|^\nu)$ with $\nu\in (0,1/2)$,} \\
 		\noalign{\vskip 1mm}
 		&\hbox{\ $\star$\  the strongly oscillating functions} \\
 		&\hbox{\qquad $\bigl[\log(3+s)\bigr]^{\sin[\log\log(3+s)]}
 			\quad\hbox{ and }\quad
 			\exp\bigl[|\log s|^\nu\cos(|\log s|^\gamma)\bigr],\ \ \nu, \gamma>0,\ \nu+\gamma<1/2$,} \\
 		\noalign{\vskip 1mm}
 		&\hbox{\ $\star$\  $1 +a \sin\bigl(\log^\nu(2+s)\bigr)$ with $ \nu\in (0,1/2)$ and $|a|<1$.}
 	\end{aligned}
 	\right.
 	$$
		 Our results thus cover a large class of non scale invariant nonlinearities. However, it so far remains an open problem what is the largest possible class of $f$ for which the conclusions of Theorem~\ref{RDT} hold.
		In particular the condition $\alpha>\frac{1}{2}$ in \eqref{sem} is required for the existence of the key energy functional 
		used in the proof (see~Lemma \ref{lem}).
 	\footnotetext{where $\log_m=\log\circ\dots\circ\log$ ($m$ times), $m\in\N^*$
 		and $K>[\exp\circ\dots\circ\exp](0)$.}
	
	\item  For radial decreasing solutions of problem \eqref{eqE2}, for a large class of regularly varying nonlinearities, the final blow-up profile and refined blow-up behavior are obtained in the companion paper \cite{chso}. Theorem \ref{RDT} of the present paper is used as an important tool in the proofs in \cite{chso}.
 \end{itemize}
 \end{rmq}

   Our next main result, for the case $\Omega=\R^n$, shows the compactness of the blowup set for initial data decaying at infinity. Note that this assumption is essentially optimal, in view of examples in \cite{GU} (see  also \cite[Remark $24.6$(ii)]{quittner2019superlinear}) of solutions blowing up at space infinity for nondecaying bounded initial data.
    To this end we recall the notation $C_0(\R^n)=\{\phi\in C(\R^n);\ \lim_{|x|\to\infty} \phi(x)=0\}$.

\begin{thm}\label{ell5}
	  Let $1<p<p_S$, assume \eqref{sem0}-\eqref{sem},  $f(0)=0$, and let  $u_0 \in C_0(\R^n)$ with $u_0\ge 0$  be such that $T<\infty$.
	Then the blow-up set of $u$ is compact. 
	  More precisely, there exists $R>0$ such that $$\underset{|x|>R,\, t\in (0,T)}{\sup} u<\infty.$$
\end{thm}

 As a consequence of Theorem \ref{RDT} we also get the following ``no-needle'' property for blow-up solutions.

\begin{prop} \label{no-needle}
	Under  the assumptions of Theorem \ref{RDT}, if  $a\in\Omega$ is a blowup point, then 
	\be{ell23}
	 \lim_{(x,t)\to(a,T)} u(x,t)=\infty. 
	\ee
\end{prop}

 	\begin{rmq}
		  The analogues of Theorems~\ref{RDT} and \ref{ell5} are
		valid for sign changing solutions
		  (assuming for instance that $f$ is an 
		   odd function on $\R$), provided the solution is known to be of type~I;
		 see Theorems~\ref{progen} and \ref{ell50} below.
	\end{rmq}

Throughout this 
paper, we shall use the notation  
\be{F}
F(X):=\int_{X}^{\infty}\frac{ds}{f(s)}.
\ee
Note that, under assumptions \eqref{sem0}-\eqref{sem}  with $p>1$,  
there exists a large $A>0$ such that $F:[A,\infty)\to (0,F(A)]$ is  well defined and decreasing. 
Moreover,  $\psi$ is defined on some interval $(T-\eta,T)$ and  we  have  
\be{psiF}
\psi(t)=F^{-1}(T-t),\quad t\in (T-\eta,T).
\ee

The rest of the paper is organizezd as follows. 
In Subsection~\ref{ell55}, we present the main ideas of our proofs.
Sections~\ref{proof-thm1} and \ref{proof-thm2} are respectively devoted to the proof of Theorems \ref{RDT} and \ref{ell5},
and of their extension to sign-changing solutions. 
Finally in Appendix, based on \cite{souplet2022universal}, we provide the necessary type~I blowup estimate for nonnegative solutions.

\subsection{Ideas  of proofs}\label{ell55}

	 To prove Theorem \ref{RDT} (and its extension to possibly sign-changing solutions
	given in Theorem~\ref{progen} below),
	we shall adapt the methods of \cite{giga1989cha,giga1985asymptotically,giga1989nondegeneracy} to equations with
	genuinely non-scale invariant nonlinearities by taking advantage of the slow variation property of $L$ in an appropriate manner. 
	Since  we want to compare $u$ with  the
	solution of ODE $\psi$ near the blowup point $a$, we  first rescale the equation by 
	{\it similarity variables and ODE renormalization} 
	by setting $y=(x-a)/\sqrt{T-t}$, $s=-\log(T-t)$ and defining the rescaled function $w_a$  by  $u(x,t)=\psi(t)w_a(y,s)$.
	 Such a rescaling, which extends the seminal idea from \cite{giga1985asymptotically} in the case $f(u)=u^p$,
was applied in \cite{duong2018construction} in the particular case of a logarithmic nonlinearity in order
to construct special blow-up solutions with prescribed profile.
	The function $w_a$ is then a global solution of 
\be{rescwa}
	\partial_sw_a-\Delta w_a+\frac{1}{2}y\cdot\nabla w_a=e^{-s}\psi^{p-1}\Big(|w|^{p-1}wL(|w|\psi)-L(\psi)w\Big)
	\ee
	 (where we omit the variables without confusion) 
		in $\mathcal{W}_a:=\{(y,s) : s_0<s<\infty,\ y\in D(s)\}$, with $ D(s):= e^{s/2}(\Omega-a)$ and 
		 $s_0>0$ large. 
	Moreover, 
	$w_a$ and $|\nabla w_a|$ are bounded under  our assumptions.
	We thus want to show that $w_a$ is attracted by the set of equilibria of \eqref{rescwa},  which 
	turns out to be the same as in the pure power case, namely $\{0,-1,+1\}$.  
	 As a significant source of difficulty, some new, nonautonomous factors
	arise from the slowly varying part of the nonlinearity. To handle them, we rewrite \eqref{rescwa} 	as the pure power plus a perturbative term,  namely
	$$ 
	\partial_sw_a-\Delta w_a+\frac{1}{2}y\cdot\nabla w_a={ \beta|w|^{p-1}w-\beta w+H(s,y)}.$$ 
	By suitably exploiting the slow variation hypothesis, we can show that $H$ 
	converges to $0$  in $L^\infty$-norm as $s\to \infty$ and 
	 is globally square integrable in space-time with respect to the Gaussian measure. This enables us to construct a 
	weighted energy $G_a(s)$, which is a modified version of that in \cite{giga1985asymptotically} and  
	can be used as a Liapunov functional to show the desired convergence to a constant steady-state.

	The nondegeneracy of blow-up (i.e.~ruling out the case $w_a\to 0$) is then obtained by a similar argument 
	as in \cite{giga1989nondegeneracy}, based on weighted energy,
	but using our modified energy functional,
	and a removable singularity property, namely a local lower bound on the blow-up rate. 
	Under our assumptions on $f$, the latter (valid for all $p>1$) takes the following form:
	if $|u(x,t)|\le \varepsilon\psi(t)$ in some neighbourhood of $(a,T)$ with $\varepsilon>0$ sufficiently small, then $a$ is not a blow-up point.  This is proved by a comparison argument  extending that from \cite[Proposition 25.1]{quittner2019superlinear} for the pure power case.
	 
	 As for Theorem \ref{ell5}, it is obtained by combining the above arguments 
	with the fact that,
	 under the assumption $u_0 \in C_0(\R^n)$, 
the weighted energy $G_a(s)$ can be shown to decay as $|a|\to\infty$.

\section{Extension and proof of Theorem \ref{RDT}} \label{proof-thm1}

 As mentioned before,  we have the following extension of Theorem \ref{RDT} for possibly sign-changing solutions
 under a type I blowup assumption.

	\begin{thm}\label{progen}
	  Let $1<p\le p_S$.  Assume that
 $f\in C^1(\R)$ is an odd function, with { $f>0$} for large $s>0$, 
and that $f$ satisfies \eqref{sem}.
		Let $u_0 \in L^\infty(\Omega)$ 
		satisfy $T:=T_{\max} (u_0)<\infty$ and \eqref{dam1}.
		If  $a\in \Omega$ is a blow-up point of $u$, then  
		\be{jojo}
		\lim_{t\to T}\frac{u(a+y\sqrt{T-t},t)}{\psi(t)}=\pm 1,
		\ee
		uniformly on compact sets $|y|\le C$.
\end{thm}

	\begin{remark}
  The oddness assumption is made here only for simplicity. Similar conclusions could be obtained under more general hypotheses.
 	\end{remark}
	
 The proof of Theorem~\ref{progen} is divided into several subsections for clarity.  
In  subsection \ref{s}, we prove a local lower  bound  
on the blow-up rate. In subsection \ref{sss}, we give an upper estimate of  $|\nabla u|$
which is  needed for the existence of the key  
weighted energy functional and its properties. 
 The latter are derived in subsection \ref{ell17}.
With the help of these tools, we then conclude
the proof of Theorem~\ref{progen} by showing 
 the convergence and the nondegeneracy of blow-up in  subsections \ref{ssss} and \ref{sssss} respectively.

\subsection{Local lower bound on the blow-up rate}\label{s}

\begin{prop}\label{LDC}
Let $f\in C^1(\R)$ be odd. Assume that there exist $q, z_0>1$ such that 
	\be{hypgq}
	\hbox{$f(z)>0$ and $g(z):=z^{-q}f(z)$ is { non-decreasing} for $z>z_0$.}
		\ee
		Let $T, \delta>0$, $a\in\mathbb{R}^n$  
		and set  
	$\mathcal{Q}:=B(a,\delta)\times(T-\delta^2,T)$. 
	 There exist $\eps_0,\delta_0>0$ depending only on $n,f$ such that if $\delta\in(0,\delta_0]$ and $u$ is a classical solution of 
	\be{L3}
	u_t-\Delta u=f(u),\quad (x,t)\in \mathcal{Q},
	\ee
	and satisfies \be{L4}
	|u(x,t)|\le \eps_0\psi(t),\quad (x,t)\in\mathcal{Q},
	\ee
	then $u$ is uniformly bounded in a neighborhood of $(a,T)$. More precisely,
	$$\sup_{|x-a|<\delta/4}|u(x,t)|\le C,\quad \hbox{ for all} \quad t\in\Bigl(T-\frac{\delta^2}{2},T\Bigr),$$	
	where $C>0$ depends only on $n,f,\delta$ and $\|u(T-\frac{\delta^2}{2})\|_{L^\infty({B_a(\delta)})}$.
\end{prop}

 \begin{rmq}\label{rempropL3}
 (i) We note that under assumption \eqref{hypgq}, the ODE $y'=f(y)$ indeed has a positive solution $\psi(t)$ on some interval $(T-\delta_0^2,T)$ which blows up at $t=T$.

\smallskip

(ii) The result in Proposition \ref{LDC} is of independent interest since it is valid for any $p>1$ 
and more general nonlinearities than Theorem~\ref{progen}.
	Note that no boundary conditions are assumed and that this is a purely local result. 
	\smallskip
	
	(iii) Proposition \ref{LDC} is the analogue of
	 \cite[Theorem 2.1]{giga1989nondegeneracy} for the pure power case.
	 See also \cite[Theorem 25.3]{quittner2019superlinear}, which provides a different proof, 
based on comparison, a quadratic change of unknown and  a cut-off.
It turns out that, by suitable modifications of the arguments in \cite{quittner2019superlinear}, we can handle 
 rather general nonlinearities without scale invariance.
 \end{rmq}
 
\begin{proof}[Proof of Proposition \ref{LDC}]
	By assumption \eqref{hypgq}, we may pick $\delta_0,C_0>0$ depending only on $f$, such that 
	\be{hypgq2}
	\psi(t)\ge z_0>1\quad\hbox{and}\quad f(\psi(t))\ge C_0\psi^q(t)\quad\hbox{ for all $t\in(T-\delta_0^2,T)$.}
	\ee
	Let $\delta\in(0,\delta_0]$.
By a space-time translation, we may assume $a = 0$ and $T = \delta^2$
and then $\mathcal{Q}=B_\delta\times(0,\delta^2)$.
Set \be{L44}
{ \mu=\min(1/2,(q-1)/4)}.
\ee
For given $R>0$, we may find $\phi\in C^2(\mathbb{R}^n)$ such that $0\le \phi\le 1$,
\be{L5}
\begin{cases}\phi(x)=0\text{ for } |x|\ge R/\sqrt{2}\\
\phi(x)=1\text{ for } |x|\le R/{2}
\end{cases}
\ee
and 
\be{L6}
|\nabla \phi|^2+|\Delta \phi^2|\le { C(n,R)}\phi^{2(1-\mu)}.
\ee

We choose $R=\delta$ and set \begin{equation*}
v=u^2\phi^2.
\end{equation*}
For $(x,t)\in \mathcal{Q}$, we have
$$v_t-\Delta v=2uu_t\phi^2-2\phi^2(u\Delta u+|\nabla u|^2)-8u\phi \nabla u\cdot\nabla\phi-u^2\Delta\phi^2.$$
Since $4|u\phi \nabla u\cdot\nabla\phi|\le\phi^2|\nabla u|^2+4u^2|\nabla \phi|^2$, therefore
\be{L9}
v_t-\Delta v\le 2 \phi^2|uf(u)|+u^2(8|\nabla\phi|^2+|\Delta \phi^2|).
\ee
Let $\varepsilon\in(0,1]$ to be chosen below. Setting $M_0=\displaystyle\sup_{|z|\le z_0}|zf(z)|$,
using  \eqref{hypgq}, \eqref{L4} and writing $|uf(u)|=|u|^{q-1}g(|u|)u^2$, we have
$$|uf(u)|\le
	\begin{cases}
	 \varepsilon^{q-1}\psi^{q-1}g(\psi)u^2& \hbox{if $|u|\ge z_0$} \\
M_0&\hbox{otherwise.} 
	\end{cases}
	$$
Note that $q-1\ge\frac{2\mu}{1-\mu}> 2\mu$ owing to \eqref{L44}. Using \eqref{L4} again and \eqref{L6}, it follows that
$$\begin{aligned}
v_t-\Delta v
&\le 2\bigl(\varepsilon^{q-1}\psi^{q-1}g(\psi)u^2+M_0\bigr)\phi^2+|u|^{2\mu}v^{1-\mu}\phi^{-2(1-\mu)}(8|\nabla\phi|^2+|\Delta\phi^2|)\\
&\le \varepsilon^{2\mu}\Bigl(2\frac{f(\psi)}{\psi}v+C(n,R)\psi^{2\mu}v^{1-\mu}\Bigr)+2M_0.
\end{aligned}$$
Next using Young's inequality and \eqref{hypgq2}, we may estimate the second term by
$$C(n,R)\psi^{ 2\mu}v^{ 1-\mu}
\le \psi^{\frac{ 2\mu}{1-\mu}}v+C^{\frac{1}{\mu}}(n,R)\le \psi^{q-1}v+C^{\frac{1}{\mu}}(n,R)
\le C_0^{-1}\frac{f(\psi)}{\psi}v+C^{\frac{1}{\mu}}(n,R).$$
Consequently, we obtain, for some $C_1=C_1(f)>0$ and $C_2=C_2(f,R,n)>0$,
\be{comment}
v_t-\Delta v\le C_1\varepsilon^{2\mu}\frac{f(\psi)}{\psi}v+C_2,\quad (x,t)\in \mathcal{Q}.
\ee

Let $\overline{v}=K\psi^\theta$ for $t\in(0,T)$, with $K,\theta>0$ to be chosen.
Using \eqref{hypgq2}, we see that
$$\overline{v}_t 
=K\theta\psi^{\theta-1}f(\psi)
 \ge  \frac{\theta}{2}\Bigl(\frac{f(\psi)}{\psi}\overline v+KC_0\Bigr).$$
Choosing $\theta= \min(2C_1\varepsilon^{2\mu},1)$ 
and ${ K=\max\Big(2C_2(C_0\theta)^{-1},{\|v(\cdot,T/2)\|_\infty}\Big)},$
 it follows that $\overline{v}$ is a supersolution to \eqref{comment} and that $\overline{v}(T/2)\ge\|v(\cdot,T/2)\|_\infty$. 
Since $v=0$ on $\partial B_\delta\times(0,T)$, we deduce from the comparison
principle that $v\le \overline{v}$ in $B_\delta\times[T/2,T)$, hence 
\be{L18}
 u^2\le K \psi^{\theta}\quad \text{in } B_{\delta/2}\times[T/2,T).
\ee
 At points $(x,t)$ where $|u|\ge z_1:=\max(z_0,K)$, inequality \eqref{L18} implies $\psi^\theta\ge K$ hence 
$|u|\le \psi^\theta\le \psi$ so that, by Young's inequality, \eqref{hypgq} and \eqref{hypgq2},
$$2|uf(u)|+C(n,R)u^2\le 3|u|^{q+1}g(|u|)+C(n,R)
\le 3\psi^{(q+1)\theta}g(\psi)+C(n,R).$$
It follows that
\be{L18b}
2|uf(u)|+C(n,R)u^2
\le 3\psi^{(q+1)\theta}g(\psi)+K_1\quad \text{in } B_{\delta/2}\times[T/2,T),
\ee
with $K_1=K_1\bigl(n,R,f,\theta,\|u(\cdot,T/2)\|_\infty\bigr)>0$.

\smallskip

Now consider $v =u^2\phi^2$ with $R=\delta/2$ instead of $R=\delta$ in \eqref{L5}.
Taking $\varepsilon=\eps(f)>0$ sufficiently small so that  $\theta=2C_1\varepsilon^{2\mu}<(q-1)/2(q+1)$,
inequalities \eqref{L9} and \eqref{L18b} imply
$$ v_t-\Delta v \le 3\psi^{(q-1)/2}g(\psi)+\tilde K_1 \quad \text{in }  B_{\delta/2}\times[T/2,T),$$ 
with $\tilde K_1=\tilde K_1\bigl(n,R,f,{\|v(\cdot,T/2)\|_\infty}\bigr)>0$.
Let  $\tilde{v}=K_2+\tilde K_1 t-\frac{6}{q-1}\psi^{(1-q)/2}$ with $K_2>0$ to be chosen. By a simple computation, we have 
$$\tilde{v}_t= 3 \psi^{-(q+1)/2}f(\psi)+\tilde K_1=3 \psi^{(q-1)/2}g(\psi)+\tilde K_1.$$
 Choosing 
$$ K_2={\|v(\cdot,T/2)\|_\infty}+\frac{6}{q-1}\psi^{(1-q)/2}(T/2)$$
and noting that $\psi'\ge 0$ on $(T/2,T)$, it follows that $v\le \tilde{v}$ on $\partial_p\Bigl(B_{\delta/2}\times[T/2,T)\Bigr)$.
Using the comparison principle once again, we obtain
$$ v\le \tilde v\le K_3:=K_2+\tilde K_1T,\quad \hbox{in }  B_{\delta/2}\times[T/2,T).$$
By the definition of $v$ and \eqref{L5}, we conclude that
 $$\sup_{|x|<\delta/4}u^2(t)\le K_3,\quad \frac{T}{2}<t<T.$$ 
\end{proof}

\subsection{Upper estimate of $|\nabla u|$} 
\label{sss}

 In this subsection we show that the upper blow-up estimate \eqref{dam1} implies a similar estimate for the gradient. The latter
 is needed for the existence and properties
of our energy functional in section \ref{ell17} below.

\begin{prop}\label{propo}
		Let  $f\in C^1(\R)$ be odd and assume 
		that there exist $m>q>1$ and $z_0>1$ such that 
	\be{hypgqm}
	\hbox{$f(z)>0$, $z^{-q}f(z)$ is { non-decreasing} and $z^{-m}f(z)$ is { non-increasing} for $z>z_0$.}
		\ee
		Let $u_0 \in L^\infty(\Omega)$ and assume that $T:=T_{\max} (u_0)<\infty$ and that $u$ satisfies \eqref{dam1} for some $M,\delta>0$.
		Then, there exists $t_0=t_0(f,T)\in(T-\delta,T)$ such that
	\be{pop}
	\|\nabla u(t)\|_\infty\le M_1 (T-t)^{-1/2}\psi(t), \quad  t_0<t<T, 
	\ee
	for some $M_1= M_1(\Omega,M,f)>0$.
\end{prop}

 \begin{rmq}\label{rempropo}
Under  assumption \eqref{sem} 
 of Theorem \ref{RDT}, property
	\eqref{hypgqm} is satisfied for any $q,m$ such that $1\le q<p<m$, owing to
 $$\bigl(z^{p-q}L(z)\bigr)'=(p-q)z^{p-q-1}L(z)+z^{p-q}L'(z)=z^{p-q-1}L(z)\Bigl(p-q+\frac{zL'(z)}{L(z)}\Bigl).$$
 \end{rmq}

The proof of Proposition \ref{propo} is based on modifications of arguments in \cite{giga1985asymptotically, quittner2019superlinear},
	suitably adapted to cover larger classes of nonlinearities.

\begin{proof}[Proof of Proposition \ref{propo}]  Let $t_1\in(\max(0,T-1),T)$ be such that $\psi(t)\ge z_0$ in $(t_1,T)$.
 We first claim that
\be{cor}
			\|f(u(t))\|_\infty\le 		 C(M,f)\frac{\psi(t)}{T-t},\quad t_1<t<T.
			\ee
		To this end, we firstly show that 
		\be{cor2}
		\sup_{X\ge k^{-1}z_0}\frac{f(X)}{X}F(kX)\le\frac{k^{ 1-m}}{q-1}\quad\hbox{ for each $k\in(0,1]$.}
		\ee 
		By  the increasing property in \eqref{hypgqm}, we note that, for all $X\ge z_0$,
		\be{cor3}
		\frac{f(X)}{X}F(X)=\frac{f(X)}{X}\int_{X}^{\infty}\frac{ds}{f(s)}=\frac{f(X)}{X}\int_{X}^{\infty}\frac{s^qds}{s^qf(s)}
		\le \frac{f(X)}{X}\frac{X^q}{f(X)}\int_{X}^{\infty}\frac{ds}{s^q}=\frac{1}{q-1}.
		\ee 
		On the other hand, 
		 the decreasing property in \eqref{hypgqm} yields $f(kz)\ge k^{ m}f(z)$ for all $z\ge k^{-1}z_0$ hence, by the change of variable $s=kz$,
		$$
		F(kX)=\int_{kX}^{\infty}\frac{ds}{f(s)}=k\int_{X}^{\infty}\frac{dz}{f(kz)}
		\le k^{ 1-m}\int_{X}^{\infty}\frac{dz}{f(z)}=k^{ 1-m}F(X),\quad X\ge k^{-1}z_0,
		$$
		and this combined with \eqref{cor3} implies \eqref{cor2}.
		Next, at points where $|u(x,t)|\ge  \max(M,1) z_0$, applying \eqref{cor2} with $k=\min(1,M^{-1})$
		and using \eqref{dam1}, the decreasing monotonicity of $F$  and \eqref{psiF}, we have 
		$$|f(u(x,t))|=\Bigl|\frac{f(u(x,t))}{u(x,t)}u(x,t)\Bigr| \le \frac{C(M,f)\psi(t)}{F(k|u(x,t)|)}
		\le \frac{C(M,f)\psi(t)}{F(\psi(t))} = \frac{C(M,f)\psi(t)}{T-t}.$$
		Since  $\psi(t)\ge 1\ge T-t$, this implies
		\begin{equation*}
		\|f(u(t))\|_\infty \le  C(M,f)\Bigl(1+\frac{\psi(t)}{T-t}\Bigr) \le2C(M,f)\frac{\psi(t)}{T-t},\quad  t_1<t<T,
		\end{equation*}
hence \eqref{cor}.

	 Let $t_1<s<t<T$.
	Denote by $(e^{-\tau A})_{\tau\ge0}$ the heat semigroup on $\Omega$ (with Dirichlet conditions for $\Omega\neq\mathbb{R}^n$).
	By the variation of constants formula, the gradient estimate for $e^{-\tau A}$ 
	(see, e.g, \cite[Proposition~48.7*]{quittner2019superlinear}) and \eqref{dam1}, 
	we have 
	\begin{align*}
	\|\nabla u(t)\|_\infty&\le \|\nabla e^{-(t-s)A}u(s)\|_\infty+\int_{s}^{t}\|\nabla e^{-(t-\tau)A}f(u(\tau))\|_\infty d\tau\\
	&\le C(\Omega)(t-s)^{-1/2}\|u(s)\|_\infty+C(\Omega)\int_{s}^{t}(t-\tau)^{-1/2} \bigl\|f(u(\tau))\bigr\|_\infty d\tau\\
	&\le C(\Omega,M)(t-s)^{-1/2}\psi(s)+C(\Omega,M,f)\int_{s}^{t}(t-\tau)^{-1/2}(T-\tau)^{-1}\psi(\tau) d\tau,\\ 
	&\le C(\Omega,M,f)\Bigl((t-s)^{-1/2}\psi(s)+\psi(t)(T-t)^{-1}\int_{s}^{t}(t-\tau)^{-1/2}d\tau  \Bigl),\\
	&\le C(\Omega,M,f)\Bigl((t-s)^{-1/2}+2(T-t)^{-1}(t-s)^{1/2}\Bigl)\psi(t).
	\end{align*}
	 Assume $t>t_0:=(t_1+T)/2$ and choose $s:=2t-T$. Since $t_1<s<t<T$ and $t-s=T-t$, it follows that
	$$	\|\nabla u(t)\|_\infty\le  3C(\Omega,M,f)(T-t)^{-1/2}\psi(t).
$$

	\end{proof}

\subsection{Weighted energy functional and its properties}\label{ell17}
   In subsections~\ref{ell17}-\ref{sssss}, we shall work under the assumptions of Theorem~\ref{progen}. 
 As mentioned in Subsection~\ref{ell55}, 
 our proof is essentially based on 
  transformation by similarity variables and ODE renormalization and on weighted energy functional, 
 after appropriate modifications of ideas in 
\cite{giga1989cha,giga1985asymptotically,giga1989nondegeneracy} (see also \cite[section 23.4]{quittner2019superlinear}). With this transformation we obtain a rescaled equation  that we rewrite (cf \eqref{VraieE}) as the pure power case  plus an 
  additional term $H$  coming from $L$. Here  new 
 difficulties arise 
 in the case of non scale invariant nonlinearities ($L\ncong1$). Under assumption \eqref{sem}, we  shall show that 
 $\|H(s,\cdot)\|_\infty\to0$ as $s\to\infty$ and  is globally square integrable in space-time with respect to the Gaussian measure, which  will guarantee the existence of  the energy and its properties (cf.~Lemma \ref{lem2}). To this end, we divide this section 
  in three parts as follows.

\smallskip

\subsubsection{Preliminaries.} 
Let $u$ be a solution of \eqref{eqE2} with blow-up time  $T\in (0,\infty)$. 
Let $t_0$ be given by Proposition \ref{propo}, and take $s_0>\max(2,-\log(T-t_0))$ 
such that $e^{\beta s}\ge s^{4\alpha}\ge 16$ for all $s\ge s_0$ and such that $\psi_1(s):=\psi(T-e^{-s})$ exists for all $s\ge s_0$ and satisfies
\be{decaypsi1}
\psi_1(s)\ge e^{\beta s/2},\quad s\ge s_0
\ee
(inequality \eqref{decaypsi1} follows from \eqref{F}, \eqref{psiF}\ and the fact that, by Remark~\ref{rempropo}, 
$f(s)\le s^{p+\eps}$ as $s\to\infty$ for any $\eps>0$).

Let $a\in \Omega$. We set 
\be{dam111}
y:=\frac{x-a}{\sqrt{T-t}},\quad s:=-\log(T-t)
\ee
and define the rescaled function
\be{dam11}
w(y,s)=w_a(y,s):=\frac{u(a+y e^{-s/2},T-e^{-s})}{\psi_1(s)}.
\ee
Note that it is equivalent to $u(x,t)=\psi(t)w(y,s)$. By a simple computation using $\psi_1'=e^{-s}f(\psi_1)$, we have  
$$
\nabla w=\frac{e^{-s/2}}{\psi_1}\nabla u,\qquad  \Delta w=\frac{e^{-s}}{\psi_1}\Delta u,
$$
and
\be{dam11b}
w_s=-\frac{y}{2}\cdot\nabla w+e^{-s}\frac{u_t}{\psi_1}-e^{-s}\frac{f(\psi_1)}{\psi_1}w,
\ee
hence $$w_s-\Delta w+\frac{1}{2}y\cdot \nabla w=\frac{e^{-s}}{\psi_1}\Bigl(u_t-\Delta u-f(\psi_1)w\Bigl).$$
Then $w$ is a global (in time) solution of 
\be{dam2}
w_s-\Delta w +\frac{1}{2}y\cdot\nabla w= h(s)\Big(|w|^{p-1}w\frac{L(\psi_1 |w|)}{L(\psi_1)}-w\Big),\quad (y,s)\in\mathcal{W}_a,
\ee
where
\be{LO}
\mathcal{W}_a:=\{(y,s) : s_0<s<\infty, \ y\in D(s)\},\quad D(s){\strut =D_a(s)}:= e^{s/2}(\Omega-a),\quad 
h(s):=e^{-s}\psi_1^{p-1}L(\psi_1).
\ee 
 Note that $\lim_\infty D(s)=\mathbb{R}^n$ and that $\mathcal{W}_a=\mathbb{R}^n\times(s_0,\infty)$ for $\Omega=\mathbb{R}^n$. 
 In term of the variables $y$ and $s$ and rescaled function $w$, we observe that \eqref{dam1} 
  and Proposition \ref{propo} imply the uniform estimates
 \be{wysM}
 |w(y,s)|\le M,
 \ee
 \be{DwysM1}
 |\nabla w(y,s)|\le M_1,
 \ee
where $M_1=M_1(\Omega,M,p,T)$, and the desired  result \eqref{jojo} can be restated as 
$$
\lim_{s\to \infty} w(y,s)=\pm 1,\ \text{uniformly on compact sets } |y|\le C.
$$
As in \cite{giga1985asymptotically}, we shall apply dynamical systems arguments to show that the global bounded solution $w$ of \eqref{dam2}
is attracted by the set of equilibria,  and it will turn out that these are the same as in the pure power case, namely they are solutions of
\be{dam39}
\Delta w-\frac{1}{2}y\cdot\nabla w-\beta w+\beta|w|^{p-1}w=0\quad \text{in }\  \mathbb{R}^n.
\ee
As a main difference
with the case $f(u)=|u|^{p-1}u$ in \cite{giga1985asymptotically}, equation \eqref{dam2} is not autonomous and the obtention of a Liapunov functional property requires additional effort.

\smallskip

\subsubsection{Control of nonautonomous terms.} 
We  rewrite equation \eqref{dam2}  as 
\be{VraieE0}
w_s-\Delta w +\frac{1}{2}y\cdot\nabla w=\beta|w|^{p-1}w-\beta w+H(s,y),
\ee
equivalently
\be{VraieE}
\rho w_s-\nabla\cdot(\rho\nabla w)=\rho\beta(|w|^{p-1}-1)w+\rho H(s,y),
\ee
where the nonautonomous term is given by
\be{VraieEE}
H(s,y):=\bigl(h(s)-\beta\bigl)\big(|w|^{p-1}-1\big)w+ h(s)|w|^{p-1}w\bigg(\frac{L(\psi_1|w|)}{L(\psi_1)}-1\bigg){\strut=:H_1(s,y)+H_2(s,y),}
\ee
and the Gaussian weight $\rho$ is defined by $\rho(y):=e^{- |y|^2/4}$.
Under  assumption \eqref{sem}
a key property for the existence of a Liapunov functional and the control of the nonautonomous terms is the following:
  
\begin{lem}\label{lem} 
Under the assumptions of Theorem~\ref{progen} with any $p>1$, we have
\be{decayH}
\|H(s,\cdot)\|_\infty\le C_0 s^{-\alpha}\log s,\quad s>s_0,
\ee
with $C_0>0$ independent of $a$.
\end{lem}

\begin{proof}[Proof of Lemma~\ref{lem}]
  It suffices to show that $H_1$ and $H_2$ satisfy \eqref{decayH}. 
To prove it for $H_1$, we first claim that the function $h(s)=e^{-s}\psi_1^{p-1}L(\psi_1)$ satisfies

\be{claimhbeta}
|h(s)-\beta|\le Cs^{-\alpha},\quad s>s_0
	\ee
(here and in the rest of the proof, $C$ denotes a generic positive constant independent of $a, s, y$).
Indeed, integrating by parts, we have
$$
F(X)=\int_{X}^{\infty}\frac{dz}{f(z)}=\int_{X}^{\infty}\frac{z^{-p} dz}{L(z)}
=\beta\frac{X^{1-p}}{L(X)}-\beta\int_{X}^{\infty}z^{1-p}\frac{L'(z)}{L^2(z)}dz
$$
hence, owing to \eqref{sem},
\be{XpLF}
\bigl|X^{p-1}L(X)F(X)- \beta\bigr|
\le CX^{p-1}L(X)\int_{X}^{\infty} \frac{z^{-p}}{\log^\alpha{\hskip -2pt}z \, L(z)}dz
\le C\frac{X^{p-1}L(X)F(X)}{\log^\alpha {\hskip -2pt}X},\quad X\ge 2.
\ee
In particular $\displaystyle\lim_{X\to\infty} X^{p-1}L(X)F(X)=\beta$ and, going back to \eqref{XpLF}, we get
$$\bigl|X^{p-1}L(X)F(X)- \beta\bigr|\le C\log^{-\alpha}{\hskip -2pt}X,\quad X\ge 2.$$
Since $F(\psi_1(s))=F(\psi(T-e^{-s}))=e^{-s}$ in view of 
\eqref{psiF}, claim \eqref{claimhbeta} follows from \eqref{decaypsi1}.
Properties \eqref{wysM} and \eqref{claimhbeta} guarantee that $H_1$ satisfies \eqref{decayH}.

\smallskip

Let us show that $H_2$ satisfies \eqref{decayH}. 
To this end we  define 
$E=\bigl\{(y,s)\in \mathcal{W}_a;\ |w(y,s)|\le s^{-\alpha}\bigr\}$. 
{ 
We first note that \eqref{sem0}-\eqref{sem} guarantees the existence of $X_0$ such that $g(X):=\frac{f(X)}{X}$ is positive and non-decreasing on $[X_0,\infty)$. 
Let $\xi\in [0,1]$. If $\xi\psi_1\ge X_0$, then 
$0\le \frac{g(\xi\psi_1)}{g(\psi_1)}\le 1$. 
Whereas, if $\xi\psi_1\le X_0$, then 
$|\frac{g(\xi\psi_1)}{g(\psi_1)}|\le \frac{C}{g(\psi_1)}\le C$. Using \eqref{claimhbeta}, $L$ even, and $|w(y,s)|\le s^{-\alpha}\le 1$ for $(y,s)\in E$, we obtain}
	\be{(3)}
	|H_2(s,y)|{ \strut =|h(s)w| 
	\Big|\frac{g(|w|\psi_1)}{g(\psi_1)}-|w|^{p-1}\Big|}\le C |w|\le C s^{-\alpha},\quad (y,s)\in E.
	\ee
	Next consider the case when $(y,s)\in \mathcal{W}_a\setminus E$. 
	Then,  recalling the definition of $s_0$  before
	\eqref{decaypsi1}, we have 
	$|w|\psi_1\ge s^{-\alpha}e^{\beta s/2}\ge e^{\beta s/4}\ge 2$.
	By assumption \eqref{sem},
	$$\Sigma(X):=\sup_{z\ge X} \Bigl|{zL'(z)\over L(z)}\Big|\le C (\log X)^{-\alpha},\quad X\ge 2.$$
	Therefore,
	$\Sigma( e^{\beta s/4})\le C (\log(e^{\beta s/4}))^{-\alpha}\le 
	C s^{-\alpha}$, as well as $M\ge |w|\ge s^{-\alpha}$.
	 Since $\min(|w|\psi_1,\psi_1)\ge e^{\beta s/4}$, it follows that
	$$
	\Bigl|\log\Bigl({L(|w|\psi_1)\over L(\psi_1)}\Bigr)\Bigr|=\Bigl|\int_{|w|\psi_1}^{\psi_1} {L'(z)\over L(z)} dz\Bigr|
	\le \Sigma\bigl( e^{\beta s/4}\bigr) 
	\Bigl|\bigl[\log z\bigr]_{|w|\psi_1}^{\psi_1}\Bigl|\ 
	\le \Sigma\bigl( e^{\beta s/4}\bigr))|\log |w||
	\le C s^{-\alpha}\log s,
	$$
	hence
	$$|H_2(s,y)|\le C\Bigl|{L(|w|\psi_1)\over L(\psi_1)}-1\Bigr|\le C s^{-\alpha}\log s.$$
	Combining with \eqref{(3)}, we deduce that
	$$\|H_2(s,\cdot)\|_\infty\le  Cs^{-\alpha}+C s^{-\alpha}\log s\le C s^{-\alpha}\log s,\quad s\ge s_0.$$
	Consequently $H_2$, and hence $H$, satisfies \eqref{decayH}.
\end{proof}

 \subsubsection{Weighted energy functional and its properties.}
{ Fix $\delta>0$ and set $\Omega_\delta=\{x\in\Omega,\, d(x,\partial\Omega)\ge\delta\}$.
 For any $a\in \Omega_\delta$,} the Liapunov functional for equation \eqref{VraieE} { with $w=w_a$}
  will be given by the  perturbed  weighted energy  defined as follows: 
\be{LO2}
G[w](s):= E[w](s)+ C_1 s^{-\gamma},
\ee
where $\gamma=\alpha-\frac12>0$,
$$E[w]:=\int_{D(s)}\bigg(\frac{1}{2}|\nabla w|^2+\frac{\beta}{2}w^2-\frac{\beta}{p+1}|w|^{p+1}\bigg)\rho dy,$$
{ and the constant $C_1 > 0$ depends on $\delta$ but not on $a\in \Omega_\delta$.
In particular, for $\Omega=\R^n$, $C_1$ is independent of $a\in\R^n$.}
The term $C_1 s^{-\gamma}$ is designed to handle the effect of the perturbation $H$ on the energy,
 making crucial use of its square integrability in space-time (cf.~\eqref{ddpsi2}), made possible by the assumption $\alpha>1/2$ in \eqref{sem}.
 This energy enjoys the following  properties{.}
\begin{lem}\label{ell7}
Under the assumptions of Theorem~\ref{progen} with any $p>1$, we have
  $G[w]\in C^1(s_0,\infty)$ and, for all $s>s_0$,
 	\be{jojo1}
	\frac{d}{ds}G[w](s)\le -\frac{1}{2}\int_{D(s)}w_s^2\rho\le 0,
		 \ee
	\be{JO}
	G[w](s)\ge 0,
	\ee
	\be{L1}
	\int_{D(s)}w^2\rho dy\le C(n,p)\big[G(w)(s)\big]^{2/(p+1)}.
	\ee
Moreover,  for each {$\delta>0$ and} $s>s_0$, 
	\be{L2}
	{\Omega_\delta\ni\strut} a\mapsto G[w_a](s)\text{ is continuous}. 
	\ee
\end{lem}
\begin{proof}
 We first check the required regularity of $w$.
 For all $0<\eps<\tau< T$, denoting $Q_{\eps,\tau}:=\Omega\times(\eps,\tau)$, it follows from standard parabolic $L^q$ and Schauder regularity  
 (see, e.g., \cite[Section~48.1]{quittner2019superlinear}) that 
 \be{utbound}
 u_t,D^2u\in BC(\bar Q_{\eps,\tau})
 \ee
 (where $BC$ denotes the set of bounded continuous functions).
 Then, setting $g:=f'(u)u_t\in L^\infty(Q_{\eps,\tau})$,
  $v:=u_t$ is the solution of the problem $v_t-\Delta v=g$ in $Q_{\eps,\tau}$ with $v=0$ on $\partial\Omega\times(\eps,\tau)$
  and $v(\eps)=u_t(\eps)$. By parabolic $L^q$ regularity, we deduce that, for all $1<q<\infty$ and $0<\eps<\tau< T$,
  $$\sup_{x_0\in\Omega} \|v_t\|_{L^q(Q_{\eps,\tau,x_0})}+\|D^2v\|_{L^q(Q_{\eps,\tau,x_0})}<\infty,$$
  where $Q_{\eps,\tau,x_0}=(\Omega\cap B_1(x_0))\times(\eps,\tau)$ (the set $Q_{\eps,\tau,x_0}$ can be replaced by $Q_{\eps,\tau}$
  and the supremum over $x_0$ omitted
   in case $\Omega$ is bounded). By interpolation inequalities it follows that
 \be{nablautbound}
 \nabla u_t\in BC(\bar Q_{\eps,\tau}).
 \ee
 We deduce from \eqref{dam11},  \eqref{dam11b}, \eqref{wysM}, \eqref{DwysM1}, \eqref{utbound} and \eqref{nablautbound} that, for all $s_1\in(s_0,\infty)$, 
 \be{jo2}
D^2w,\ (1+|y|)^{-1}w_s,\ (1+|y|)^{-1}\nabla w_s\in BC(\overline{\mathcal{W}}_{a,s_1}),
\ee
where  $\mathcal{W}_{a,s_1}=\mathcal{W}_{a}\cap(\R^n\times(s_0,s_1))$.
In view of the exponential decay of $\rho$, this guarantees the convergence and the differentiability of the
various integrals and justifies the integrations by parts in the rest of the proof.

\smallskip

	 We next compute the variation of the first part $E[w](s)$ of the energy
	and derive a differential inequality for the weighted $L^2$ norm.
	 This is essentially the same argument as in \cite{giga1989cha,giga1985asymptotically}, 
	 but we give details for completeness and convenience.
	For all $s>s_0$, we have for $q\ge2$ \be{LOOOOOOO}
	\frac{1}{q}\frac{d}{ds}\int_{D(s)}|w|^q\rho dy=\int_{D(s)}\rho w_s|w|^{q-2}wdy+\frac{1}{ 2q}\int_{\partial D(s)}|w|^q\rho 
 (y\cdot\nu)d\sigma,
	\ee
	 where $d\sigma$ denotes the surface measure on $\partial D(s)$ 
	 and $\nu$ the exterior unit normal on $\partial D(s)$.
Since $w=0$ on $\partial D(s)$, the boundary term vanishes, hence
	\be{LDC1303242}
	\displaystyle\frac{1}{q}\frac{d}{ds}\int_{D(s)}|w|^{q}\rho dy=\int_{D(s)}\rho w_s|w|^{q-2}wdy.
	\ee
	 Next, to compute the variation of the term involving $\nabla w$, we integrate by parts to get
	$$\begin{aligned}
	\frac{d}{ds}\int_{D(s)}|\nabla w|^2\rho
	&=2\int_{D(s)} \nabla w_s\cdot(\rho\nabla w)+\frac12\int_{\partial D(s)}|\nabla w|^2\rho(y\cdot\nu)d\sigma\\
	&=-2\int_{D(s)} w_s\nabla\cdot(\rho\nabla w)+2\int_{\partial D(s)}\rho w_s(\nabla w\cdot\nu) d\sigma+\frac12\int_{\partial D(s)}
	\rho|\nabla w|^2(y\cdot\nu)d\sigma,
	\end{aligned}$$
Using  that, on $\partial D(s)$, we have
	$w_s=-\frac{y}{2}\cdot\nabla w$ 
	(owing to \eqref{dam11b}) 
	and $\nabla w=(\nabla w\cdot\nu)\nu$, hence
	$w_s(\nabla w\cdot\nu)=-\frac12|\nabla w|^2(y\cdot\nu)$, this yields
		\be{LDC1303242b}
		\frac{d}{ds}\int_{D(s)}|\nabla w|^2\rho=-2\int_{D(s)} w_s\nabla\cdot(\rho\nabla w)
	-\frac12\int_{\partial D(s)}\rho|\nabla w|^2(y\cdot\nu)d\sigma.
\ee
			Also, we have $\partial D(s)=e^{s/2}\partial(\Omega-a)$ hence $\int_{\partial D(s)}d\sigma\le C(\Omega)e^{s/2}$ and, since $a\in\Omega_{\delta}$, 
			we have $|y|\ge {\delta}e^{s/2}$ on $\partial D(s)$. 
			Consequently, using \eqref{DwysM1} { and $|y|\rho \le Ce^{-|y|^2/8}$}, we get
			$$\int_{\partial D(s)}\rho|\nabla w|^2|y\cdot\nu|d\sigma=CM_1^2{\exp\bigl[-(\delta e^{s/2})^2/8\bigr]}\le C\exp\bigl(-{\textstyle\frac18\delta^2}e^s\bigr).$$ 
	Combining this with \eqref{VraieE}, \eqref{LDC1303242}, \eqref{LDC1303242b} 
	we obtain
	\begin{align}
	\frac{d}{ds}E[w](s)&=-\int_{ D(s)}w_s\Bigl(\nabla\cdot(\rho\nabla w)-\beta\rho w+\beta\rho |w|^{p-1}w\Bigl)
	\strut-\frac12\int_{\partial D(s)}\rho|\nabla w|^2(y\cdot\nu)d\sigma\notag\\
	&\le-\int_{D(s)} w_s^2\rho+\int_{D(s)}w_sH(s,y)\rho \strut+C\exp\bigl(-{\textstyle\frac18\delta^2}e^s\bigr). \label{ddEws} 
\end{align}	
	On the other hand, setting $\Psi(s):=\int_{\mathbb{R}^n}w^2\rho dy$ and using 
	 \eqref{LOOOOOOO} with $q=2$,  \eqref{VraieE} and  integration by parts, we get
	$$\begin{aligned}
	\frac12\Psi'(s)&=\int_{D(s)}\rho w_swdy=\int_{ D(s)}\Big(\nabla\cdot(\rho\nabla w)+\beta\rho |w|^{p-1}w-\beta \rho w+\rho H(s,y)\Big)w\\
	&=\int_{ D(s)}\big(-|\nabla w|^2+\beta|w|^{p+1}-\beta w^2\big)\rho+\int_{ D(s)} w H(s,y)\rho dy \\
		&= -2E[w](s)+\frac{1}{p+1}\int_{D(s)}|w|^{p+1}\rho dy+\int_{D(s)} wH(s,y)\rho dy
			\end{aligned}	$$
	hence, by Jensen's inequality,
		\be{ddpsi}
		\Psi'(s)\ge  -4E[w](s) \strut-2M\int_{D(s)} |H(s,y)|\rho dy+C(n,p)\Psi^{(p+1)/2}(s).
		\ee
		
			We shall now make use of the key decay property in Lemma~\ref{lem} to handle the terms involving the perturbation $H$. By \eqref{ddEws}, we have
\be{ddEws2}
\frac{d}{ds}E[w](s)\le -\frac{1}{2}\int_{ D(s)}w_s^2\rho+\frac{1}{2}\int_{D(s)}H^2(s,y) \rho 
	\strut+C \exp\bigl(-{\textstyle\frac18\delta^2}e^s\bigr) 
	\ee
	and \eqref{decayH} and $\alpha>\frac12$ guarantee that, for some { $C_1>0$ depending on $\delta$ but not on $a\in \Omega_\delta$,} 
	\be{ddpsi2}
	\frac{1}{2}\int_{D(s)}H^2(s,y)\rho dy+C\exp\bigl(-{\textstyle\frac18\delta^2}e^s\bigr) 
	\le \gamma C_1 s^{-\alpha-\frac12}
	\quad\hbox{and}\quad
	2M\displaystyle\int_{D(s)} |H(s,y)|\rho dy\le C_1s^{-\gamma},
	\ee
	for all $s\ge s_0$.
It then follows from  \eqref{LO2}, \eqref{ddpsi} and \eqref{ddEws2} that
$$\frac{d}{ds}G[w](s)=\frac{d}{ds}E[w](s)-\gamma C_1 s^{-\alpha-\frac12}\le -\frac{1}{2}\int_{ D(s)}w_s^2\rho,$$
	i.e.~\eqref{jojo1}, and
$$\Psi'(s)\ge -4G[w](s)+C(n,p)\Psi^{(p+1)/2}(s).$$
 The latter, combined with \eqref{jojo1}, implies
$$\Psi'(s)\ge-4 G[w](s_1)+C(n,p)\Psi^{(p+1)/2}(s),\quad s\ge s_1\ge s_0,$$
 This guarantees \eqref{JO} and \eqref{L1}, since otherwise $\Psi$ has to blow up in finite time. 
 
 \smallskip
 
  Finally, changing variables to write  
$$ \begin{aligned}
E[w_a(s)]
&=e^{s/2}\int_{\Omega}\bigg(e^{-s/2}\Bigl|\frac{\nabla u(x,T-e^{-s})}{\psi_1(s)}\Bigl|^2\\
&\qquad\qquad +\frac{\beta}{2}\Bigl|\frac{u(x,T-e^{-s})}{\psi_1(s)}\Bigl|^2-\frac{\beta}{p+1}\Bigl|\frac{u(x,T-e^{-s})}{\psi_1(s)}\Bigl|^{p+1}\bigg)\rho((x-a)e^{s/2})dx,
\end{aligned}$$
property \eqref{L2} follows from \eqref{wysM}, {\eqref{DwysM1}} %%
%\eqref{claimhbeta} 
and dominated convergence.
\end{proof}

\subsection{Convergence}\label{ssss}

 Denote the set of bounded steady-states by
$$\quad\mathcal{S}=\bigl\{z\in C^2\cap L^\infty(\mathbb{R}^n);\ z \text{ is a solution of } \eqref{dam39}\bigr\}.$$
We first recall the classification result from \cite{giga1985asymptotically}.

\begin{prop}\label{lem2}
	 If $1<p\le p_S$, then $\mathcal{S}=\{0,1,-1\}$.
	\end{prop}

 The next lemma shows that $w$ converges to $\mathcal{S}$ as $s\to \infty$.

\begin{lem}\label{jo3}
 Under the assumptions of Theorem~\ref{progen}, for any $a \in \Omega$, there exists $\ell\in\{0,1,-1\}$ such that 
$$
	\lim_{s\to\infty}w_a(y,s)=\ell,
$$
uniformly for $|y|$ bounded.
Moreover, we have $\lim_{s\to\infty} G(w_a(s))=0$ if $\ell=0$ and 
$\lim_{s\to\infty} G(w_a(s))=\eta(n,p)>0$ otherwise.
\end{lem}

\begin{proof}
 By \eqref{jojo1}-\eqref{JO}, we have
\be{ellGw}
  l:=\lim_{s\to\infty} G[w](s)\in [0,\infty).
\ee
Pick any sequence $s_j\to \infty$ and set $z_j:=w(y,s+s_j)$.
It follows from \eqref{wysM}, \eqref{DwysM1}, \eqref{VraieE0}, \eqref{decayH}
 and parabolic estimates that
the sequence $(z_j)_j$ is  
{ bounded in $W^{2,1;q}_{loc}(Q)$ for any $1<q<\infty$, where $Q=\mathbb{R}^n\times[0,1]$.}
Consequently, there exists a subsequence (still denoted $s_j$) and  a function $z$
such that $w(\cdot,\cdot+s_j)\to z$ { weakly in $W^{2,1;q}_{loc}(Q)$ and strongly in $C^{\alpha,\alpha/2}_{loc}(Q)$ for any $\alpha\in(0,1)$.} 
 In view of \eqref{decayH} in Lemma~\ref{lem}, it follows that $z$ is a { strong} solution of
\be{limeqn}
 { \partial_sz-\Delta z-\frac{1}{2}y\cdot\nabla z-\beta z=\beta|z|^{p-1}z\quad \text{in }\  Q.}
\ee
{ By Schauder parabolic regularity, $z$ is then a classical solution of \eqref{limeqn}.} 
Moreover $z$ and $|\nabla z|$ are bounded in $\mathbb{R}^n\times[0,1]$. Let $R>0$. 
 Using \eqref{jojo1}, \eqref{ellGw} and the fact that $B_R\subset D(s)$ for all sufficiently large $s$, we deduce that
 \begin{equation*}
	\int_{0}^{1}\int_{B_{R}}\big(\partial_s z_j\big)^2\rho dy ds\le 
	 \int_{s_j}^{s_j+1}\int_{D(s)}\big(\partial_s w\big)^2\rho dyds
	\le 2G[w](s_j)-2G[w](s_j+1)\to 0,\quad j\to \infty.
	\end{equation*}
By Fatou's lemma,  it follows that $\partial_s z=0$ in $B_R$ and, since $R>0$ is arbitrary, 
 we deduce from \eqref{limeqn} that
$z\in \mathcal{S}$.  Since $\mathcal{S}$ is discrete,
the first assertion follows from an immediate connectedness argument. For the second assertion, we assume that $w(\cdot,s_j)\to 0$ or $\pm1$ uniformly for $|y|$ bounded. Using dominated convergence theorem  and \eqref{wysM}, \eqref{DwysM1}, we have $G[w(s_j)]\to 0$ if $w(s_j)\to 0$ and when $w(s_j)\to \pm 1$ we have
\begin{equation*}
	G[w(s_j)]\underset{j\to \infty}{\longrightarrow}\bigg(\int_{\mathbb{R}^n}\rho dy\bigg)\bigg(\frac{\beta}{2}-\frac{\beta}{p+1}\bigg)=\frac{(4\pi)^{n/2}}{2(p+1)}=:\eta(n,p)>0.
	\end{equation*}
	The assertion then follows from the  monotonicity of $G[w](s)$.
\end{proof}

\subsection{Nondegeneracy of blow-up and  proof of Theorems~\ref{RDT} and \ref{progen}} \label{sssss}

 In this paragraph, we shall complete the proof of Theorems~\ref{RDT} and \ref{progen} by showing that $0\notin \omega(w_a)$ if $a$ is a blowup point of $u$, 
 \begin{lem}\label{LOO}
   Under the assumptions of Theorem~\ref{progen},
	if $0\in \omega(w_a)$, then $a$ is not a blow-up point.
\end{lem}

\begin{proof}
 Let $\theta=1/(n+2)$ if $n\ge 2$ and $\theta=1/2$ if $n=1$.
We shall use the interpolation inequality  (cf.~\cite[p.250]{quittner2019superlinear}):
\be{L19}
|v(0)|\le C(n)\Big[\|v\|_{L^2(B_1)}^\theta\|\nabla v\|_{L^\infty(B_1)}^{1-\theta}+\|v\|_{L^2(B_1)}\Big],\quad v\in C^1(\bar{B_1}).
\ee
 Set $d=\frac12 {\rm dist}(a,\partial\Omega)$ and $\bar s_0=\max(s_0,-2\log d)$. Let $b\in B_d(a)$. 
For all $s>\bar s_0$, since ${\rm dist}(b,\partial\Omega)\ge d$, we have $B_1\subset D_b(s)$ hence, by \eqref{L1},
\be{L20}
\|w_b(\cdot,s)\|^2_{L^2(B_1)}\le C(n)\int_{D_b(s)} w_b^2\rho dy \le C(n,p)\big(G[w_b](s)\big)^{2/(p+1)},\quad s\ge \bar s_0.
\ee
Combining \eqref{DwysM1}, \eqref{jojo1}, \eqref{L19} and \eqref{L20}, 
we obtain
$$
|w_b (0,s)|\le C(n,p)\bigg[M_1^{1-\theta}\big(G[w_b](s_1)\big)^{\theta/(p+1)}
+\big(G[w_b](s_1)\big)^{1/(p+1)}\bigg], \quad s\ge s_1>\bar s_0.
$$
 Therefore, for $\eps_0=\eps_0(n,f)>0$ given by Proposition \ref{LDC}, there exists $\eps_1=\eps_1(n,f,M_1)>0$ such that
\be{done2}
G[w_b](s_1)\le \eps_1\ \text{for some $s_1>\bar s_0$}\quad\Longrightarrow\quad |w_b(0,s)|\le \eps_0\ \text{for all $s\ge s_1$}.
\ee
Assume that $0\in \omega(w_a)$.  Then, by the second part of Lemma \ref{jo3}, there exists $s_1>\bar s_0$ such that $G[w_a](s_1)<\eps_1$ and, 
by the continuous dependence property \eqref{L2}, there exists $d_1\in(0,d)$ such that $G[w_b(s_1)]<\eps_1$ for all $b\in B_{d_1}(a)$.
It follows from \eqref{done2} that $|w_b(0,s)|\le \eps_0$ for all $b\in B_{d_1}(a)$ and all $s\ge s_1$.
Going back to original variables, we thus have
$|u(b,t)|\le\eps_0 \psi(t)$ for all $(b,t)$ sufficiently close to $(a,T)$. By Proposition \ref{LDC}  we conclude that $a$ is not a blow-up point.
\end{proof}

 \begin{proof}[Proof of Theorem~\ref{progen}]
 Since, by \eqref{dam11}, 
$$\frac{u(a+y\sqrt{T-t},t)}{\psi(t)}= w(y,-\log(T-t)),$$
the result is a direct consequence of Lemmas \ref{jo3} and \ref{LOO}.
\end{proof}

\begin{proof}[Proof of Theorem~\ref{RDT}]
By Theorem~\ref{A1709241}, we know that $u$ is of type~I, i.e.~\eqref{dam1} holds for some $M, \delta>0$.  
Also, when $u_0\ge 0$ and $f$ is only defined for $s\ge 0$ and satisfies \eqref{sem0},
it is easily checked that all the arguments in the proof of Theorem~\ref{progen} remain valid
 (cf.~Propositions~\ref{LDC}-\ref{propo} and Lemmas~\ref{lem}-\ref{LOO}).
Consequently, Theorem \ref{RDT} follows.
\end{proof}

\subsection{Proof of Proposition~\ref{no-needle}}
 We here prove the ``no-needle'' property \eqref{ell23}.
	By a space translation, we may assume $a = 0$. 
	Then,  from Theorem~\ref{progen},  there exists $ T_1\in(0,T)$ such that 
	\be{LDC20243}
	\frac{u(x,t)}{\psi(t)}\ge\frac{1}{2},\quad  T_1\le t<T,\ |x|\le \sqrt{T-t}.
	\ee
	 Fix any $t_0\in [T_1,T)$ and let $\delta=\sqrt{T-t_0}$.
	Denoting by $\varphi$ the first eigenfunction of the negative Dirichlet Laplacian in $B_{1}$ normalized by $\|\varphi\|_\infty=1$ and $\lambda_1>0$ the corresponding eigenvalue, we set $\phi_\delta(x)=\varphi(x/\delta)$ and $\lambda_\delta=\frac{\lambda_1}{\delta^2}$.
		 On the other hand, by assumption~\eqref{sem0} there exists $C=C(f)>0$ such that $f(s)\ge -Cs$ for all $s\ge 0$.
	Let $v$ be solution of \be{LDC20244}
	\begin{cases}\begin{aligned}
	v_t-\Delta v&= -Cv\quad \quad\text{in } B_\delta\times(t_0,\infty),\\
	v(x,t)&=0\quad \quad\text{on } \partial B_\delta\times(t_0,\infty),\\
	v(x,t_0)&=\frac{1}{2}\psi(t_0)\phi_\delta(x)\quad \text{in } B_\delta.
	\end{aligned}
	\end{cases}
	\ee 
	By \eqref{LDC20243} we have $u(\cdot,t_0)\ge \frac{1}{2}\psi(t_0)\ge \frac{1}{2}\psi(t_0)\phi_\delta(x)$ in $B_\delta$. 
	The comparison principle then guarantees that $u\ge v$ in $B_\delta\times [t_0,T)$. We note that the explicit 
	solution of \eqref{LDC20244} is given by $v(x,t)=\frac{1}{2}\psi(t_0)\phi_\delta(x)\exp( -Ct-\lambda_\delta(t-t_0)).$ 
	 Denoting $\eta=\min_{\bar B_{1/2}}\varphi>0$, we obtain
	$$
	u(x,t)\ge \frac{\eta}{2}\psi(t_0)\exp\Bigl(-CT-\frac{\lambda_1}{\delta^2}(t-t_0)\Bigr)
	\ge \frac{\eta}{2}\psi(t_0)\exp(-CT-\lambda_1),\quad |x|\le\frac{\delta}{2},\ t_0\le t<T.
	$$
	Since $t_0$ can be taken arbitrarily close to $T$ and $\lim_{t\to T}\psi(t)=\infty$, property \eqref{ell23} follows.

\section{Extension and proof of Theorem \ref{ell5}} 
\label{proof-thm2}

	   Theorem \ref{ell5} will be a consequence of the following result for possibly sign-changing solutions.
	 
\begin{thm}\label{ell50}
	  Let $1<p\le p_S$.  Assume that
\be{hypell50}
		 \hbox{ $f\in C^1(\R)$ is an odd function, with ${ f>0}$ for large $s>0$,} %and $C^2$ 
		 \ee and that $f$ satisfies \eqref{sem}. 
Let  $u_0 \in C_0(\R^n)$ satisfy $T<\infty$ and \eqref{dam1}.
		Then the blow-up set of $u$ is compact.  
		 More precisely, there exists $R>0$ such that 
		\be{BUcompact}
		\underset{|x|>R,\, t\in (0,T)}{\sup} |u|<\infty.
		\ee
		 Moreover the conclusion remains valid if $u_0\ge 0$ and, instead of \eqref{hypell50},
		$f$ satisfies \eqref{sem0} and $f(0)=0$.
\end{thm}
 
Under the assumption $u_0 \in C_0(\R^n)$, we shall show that,
 at a suitable shifted time $s_1$, the weighted energy functional $G_a[w](s_1)$ (cf.~\eqref{LO2})
becomes small enough for large $|a|$.
 Then we can conclude by using the nondegeneracy analysis in section~\ref{sssss}.
The time shift, which is made necessary by the second term $C_1 s^{-\gamma}$ of the energy,
will be handled by means of the following lemma.

\begin{lem}\label{gen0}
	Let $0<t_1<T$. 
	Under assumptions of Theorem \ref{ell50}, with $p>1$, we have 
	$$|\nabla u(x, t_1)| 
	 +|u(x,t_1)| \to 0,\quad \text{as } |x|\to \infty.$$
\end{lem}

  The proof of Lemma \ref{gen0} relies on the following simple decay propagation property for the heat semigroup
$S(t)$ in $\R^n$. 
\begin{lem}\label{propopo}
For all $v\in C_0(\mathbb{R}^n)$, we have
$$\lim_{R\to\infty} \Bigl(\ \sup_{t>0,\ |x|>R} \bigl[S(t)|u_0|\bigr](x)\Bigr)=0.$$
\end{lem}

\begin{proof}[Proof of Lemma \ref{propopo}]
	 Let $v\in C_0(\mathbb{R}^n)$ and $\eps>0$. Fix $R>0$ such that $\underset{|x|\ge R}{\sup} |v(x)|\le \eps$. 
	   We first claim that
	  \begin{equation}\label{ell18}
		\sup_{|x|\ge 2R} |(S(t)v)(x)|\le \varepsilon+C_1\|v\|_\infty e^{-\frac{C_2R^2}{t}},\quad t>0.
		\end{equation}
		   Indeed, we can write
		\begin{align*}
		\bigl|\big(S(t)v\big)(x)\bigl|&=(4\pi t)^{-n/2}\Bigl|\int_{\mathbb{R}^n} e^{-\frac{|x-y|^2}{4t}}v(y)dy\Bigl|\\
		&\le \varepsilon \int_{|y|>R} (4\pi t)^{-n/2}e^{-\frac{|x-y|^2}{4t}}dy+\|v\|_\infty (4\pi t)^{-n/2}\int_{|y|\le R} e^{-\frac{|x-y|^2}{4t}}dy=:\eps J_1+\|v\|_\infty J_2.
		\end{align*}
		   We have $J_1\le1$ and, for $|x|>2R$,  $|y|<R$ implies $|x-y|>R$, so that
$$J_2\le (4\pi t)^{-n/2}\int_{|\xi|> R} e^{-\frac{|\xi|^2}{4t}}d\xi \le \int_{|z|>\frac{R}{\sqrt{t}}}e^{-\frac{|z|^2}{4}}dz\le C_1e^{-C_2R^2/t},$$
		 hence \eqref{ell18}.
		 
		 Next, observing that $|v|\le \varepsilon+\|v\|_\infty \chi_{B_R}$ and $v_0:=\|v\|_\infty \chi_{B_R}\in L^1$, 
		 it follows that, for $t_0=t_0(\eps,v)>0$ large enough, we have
  \be{ell18a}
  |(S(t)v)(x)|\le (S(t)|v|)(x)\le \varepsilon+(4\pi t)^{-n/2}\|v_0\|_1\le 2\eps,\quad x\in\R^n,\ t\ge t_0.
  \ee
   Now taking $R_0=R_0(\eps,v)>2R$ large enough, we deduce  from \eqref{ell18} that
  \begin{equation}\label{1ell18}
|(S(t)v)(x)|\le 2\eps,\quad 0<t\le t_0,\ |x|\ge R_0,
		\end{equation}
		   and the conclusion follows by combining \eqref{ell18a} and \eqref{1ell18}.
\end{proof}

\begin{proof}[Proof of Lemma \ref{gen0}]
By the representation of $S(t)$ by the heat kernel $(4\pi t)^{-n/2}\exp(-|x|^2/(4t))$, we see that
\be{gradkernel}
\bigl|[\nabla_x S(t)v](x)\bigr|\le C(n)t^{-1/2} \bigl[S(2t)|v|\bigr](x),\quad v\in L^\infty(\Omega),\ t>0,\ x\in\R^n.
\ee

Fix $t_1<T$.   Recalling $f(0)=0$, there exists $K\strut =K(t_1)>0$ such that 
$-K|u|\le u_t-\Delta u=f(u)\le K|u|$ in $\mathbb{R}^n\times(0,t_1]$. 
By the comparison principle it follows that
\be{gradkernel2}
|u|\le e^{Kt}S(t)|u_0|\le CS(t)|u_0| \quad\hbox{in $\mathbb{R}^n\times(0,t_1]$.}
\ee
(Here and in  the rest of the proof, $C$ denotes a generic positive constant, possibly depending on $t_1$.)
	Using \eqref{gradkernel}, \eqref{gradkernel2} and the variation of constant formula 
	$u(t_1)=S(t_1)u_0+\int_{0}^{t_1}S(t_1-\tau)f(u(\tau))d\tau$,
	we deduce that,  pointwise in $\R^n$,
	$$\begin{aligned}
	|\nabla  u(t_1)|&\le Ct_1^{-1/2}S(2t_1)|u_0|+C\int_{0}^{t_1}(t_1-\tau)^{-1/2} S(2(t_1-\tau))S(\tau)|u_0|d\tau\nonumber\\
	&=Ct_1^{-1/2}S(2t_1)|u_0|+C\int_{0}^{t_1}(t_1-\tau)^{-1/2} S(2t_1-\tau)|u_0|d\tau
	\le  C\bigl(t_1^{-1/2}+t_1^{1/2}\bigr)\sup_{t\in[0,2t_1]} S(t)|u_0|
	\end{aligned}$$
	hence, using again \eqref{gradkernel2},
	$$|u(t_1)|+|\nabla  u(t_1)|	\le  C\bigl(1+t_1^{-1/2}+t_1^{1/2}\bigr)\sup_{t\in[0,2t_1]} S(t)|u_0|.$$
	This combined with Lemma \ref{propopo} concludes the proof.
\end{proof}

\begin{proof}[Proof of Theorem \ref{ell50}] 
 Recall the definition \eqref{LO2} of the weighted energy:
		\begin{equation*}
	G[w_a(s)]=E[w_a](s)+C_1 s^{-\gamma},\quad s>s_0,
		\end{equation*}
		 for all $a\in\R^n$. Also, for $\eps_0=\eps_0(n,f)>0$ given by Proposition \ref{LDC}, by \eqref{done2}
		 there exists $\eps_1=\eps_1(n,f,M_1)>0$ such that
\be{done2b}
G[w_a](s_1)\le \eps_1\ \text{for some $s_1>s_0$}\quad\Longrightarrow\quad |w_a(0,s)|\le \eps_0\ \text{for all $s\ge s_1$}
\ee
(noticing that $\bar s_0=s_0$ for $\Omega=\R^n$).
Choose $s_1>s_0$ such that $C_1 s_1^{-\gamma}<\eps_1/2$.
Rewriting $E[w_a(s_1)]$ as
$$
	E[w_a(s_1)]=\int_{\mathbb{R}^n}\bigg(\frac{e^{-s_1}|\nabla u(a+e^{-\frac{s_1}{2}}y,t_1)|^2+\beta |u(a+e^{-\frac{s_1}{2}}y,t_1)|^2}{2\psi^2(t_1)}-\frac{\beta|u(a+e^{-\frac{s_1}{2}}y,t_1)|^{p+1}}{\psi^{p+1}(t_1)}\bigg)\rho(y)dy
$$
with $t_1=T-e^{-s_1}$,
and using Lemma \ref{gen0} and dominated convergence, there exists $R>0$ such that,
	for all $a$ such that $|a|\ge R-1$, we have
	$E[w_a](s_1)<\frac{\eps_1}{2}$, hence $G[w_a](s_1)<\eps_1$.
	By \eqref{done2b}, we deduce that
$$ |w_a(0,s)|\le \eps_0,\quad \hbox{for all $s\ge s_1$ and $|a|>R-1$}, $$
i.e.
$$|u(a,t)|\le\eps_0 \psi(t),\quad \hbox{for all $t\in(t_1,T)$ and $|a|>R-1$}. $$
Applying Proposition \ref{LDC} with $\delta=\min(\delta_0,1,\sqrt{T-t_1})$. It follows that
$$\underset{|x|>R,\ t\in (T-\delta^2/2,T)}{\sup} |u|<\infty,$$
hence \eqref{BUcompact}.
 It is easily checked that the final assertion (for $u_0\ge 0$) follows from the same argument.
\end{proof}

\begin{proof}[Proof of Theorem \ref{ell5}]
It follows directly from Theorems~\ref{ell50} and \ref{A1709241}. 
\end{proof}

\section{Appendix: Type-I blowup}

We have used the following result, which in particular guarantees
the type~I blowup estimate~\eqref{dam1} under the assumptions of Theorem~\ref{RDT}
(and the constant $M$ is actually independent of the solution, 
although this fact is not used in our proofs).

\begin{thm}\label{A1709241}
	Let $\Omega$ be a uniformly $C^2$ domain of $\R^n$, $f\in C([0,\infty))$ be positive for $s>0$ large
	and assume that $f$ has regular variation at $\infty$ with index $p\in (1,p_S)$.
	For each $\tau>0$, there exist $M=M(\Omega,f)>0$ and $t_0=t_0(\tau,f)\in(0,\tau)$ such that,
	if $u\ge 0$ is a strong solution of 
$$\begin{cases}
u_t-\Delta u=f(u),&x\in \Omega,\ t_0<t<\tau,\\
u=0,& x\in \partial\Omega,\ t_0<t<\tau,
\end{cases}
$$
then
\be{typeI}
u(x,t)\le MF^{-1} (\tau-t)\quad \hbox{in } \Omega\times [t_0,\tau).
\ee
\end{thm}

Theorem~\ref{A1709241} was essentially established in \cite[Theorem~3.1]{souplet2022universal}.
We note that the assumptions on~$f$, which are those from \cite{souplet2022universal}, 
are more general than those in Theorem~\ref{RDT} of the present paper (see after \eqref{prop}).
On the other hand, the estimate in \cite{souplet2022universal} is given there in a different form (cf.~\eqref{typeIa} below).
We therefore provide a proof of Theorem~\ref{A1709241}, where we derive \eqref{typeI} as a consequence of \eqref{typeIa}.

\begin{proof} 
Under the assumptions of Theorem~\ref{A1709241}, we know from \cite[Theorem~3.1 and Remark~1(i)]{souplet2022universal} that
\be{typeIa}
\frac{f(u(x,t))}{u(x,t)}\le \frac{C_0}{\tau-t},
\quad\hbox{for all $(x,t)\in \Omega\times [\tau/2,\tau)$ such that $u(x,t)\ge 1$,}
\ee
with $C_0=C_0(\Omega,f)>0$.
Also, there exist $\delta,A>0$ such that $f\ge 1$ on $[A,\infty)$ and 
\be{typeIa2}
\hbox{$F$ is a decreasing bijection from $[A,\infty)$ to $(0,\delta]$.}
\ee
We shall show that \eqref{typeIa} implies \eqref{typeI}.

To this end, we first claim  that there exist $k>1$ and $s_0>kA$ such that
\begin{equation}\label{A1708242}
 F\Bigl(\frac{s}{k}\Bigr)\ge\frac{C_0s}{f(s)}, \quad\hbox{for all $s\ge s_0$.}
	\end{equation}
	Indeed, by the representation theorem for slowly varying functions (see~\cite[Theorem 1.2]{seneta}),
	 there exist continuous functions $\tau,\xi$ such that
	$$f(s)=\tau(s)f_0(s),\quad\hbox{where } f_0(s):=s^p\exp\Big[\int_{1}^{s}\frac{\xi(z)}{z}dz\Big],\quad\hbox{for all $s\ge A$,}$$
	with $\lim_{s\to \infty}\tau(s)=\ell>0$ and $\lim_{s\to \infty}\xi(s)=0$.
	Moreover there exist $C_1,C_2>0$ such that $C_1\le \tau(s)\le C_2$ for all $s\ge A$.
	Fixing $1<m<p<q$, we see that there exists $s_1>A$ such that $s^{-m}f_0(s)$ 
	is increasing and $s^{-q}f_0$ is decreasing on $[s_1,\infty)$. For $z\ge s_1$, we have  
$$F(z)\ge C_2^{-1}\int_{z}^{\infty}\frac{ds}{f_0(s)}=C_2^{-1}\int_{z}^{\infty}s^{-q}\frac{s^q}{f_0(s)}ds
\ge C_2^{-1}\frac{z^q}{f_0(z)}\int_{z}^{\infty}s^{-q}ds=\frac{1}{C_2(q-1)}\frac{z}{f_0(z)}=:\frac{C_3z}{f_0(z)}.$$
	Let $k>1$. In view of this control of $F(z)$, we have, for all $s>ks_1$,
$$F(s/k)\ge \frac{C_3s}{kf_0(s/k)}=C_3 \Big(\frac{s}{k}\Big)^{1-m}\frac{\big(\frac{s}{k}\big)^m}{f_0(s/k)}
	\ge C_3 \Big(\frac{s}{k}\Big)^{1-m}\frac{s^m}{f_0(s)}=C_3k^{m-1}\frac{s}{f_0(s)}\ge C_3C_1k^{m-1}\frac{s}{f(s)}.$$
	Choosing $k>1$ large enough, so that $C_3C_1k^{m-1}\ge C_0$, and then $s_0= ks_1$, claim \eqref{A1708242} follows.

	 \smallskip
	 
Now set $t_0:=\max(\tau-\delta,\tau/2)$. For any $(x,t)\in \Omega\times[t_0,\tau)$ such that $u(x,t)\ge s_0$, it follows from \eqref{typeIa}, \eqref{A1708242} that
$$F\Bigl(\frac{u(x,t)}{k}\Bigr)\ge C\frac{u(x,t)}{f(u(x,t))}\ge \tau-t,$$
and \eqref{typeIa2} then implies $u(x,t)\le kF^{-1}(\tau-t)$.
Consequently, we have $u(x,t)\le s_0+kF^{-1}(\tau-t)$ in $ \Omega\times[t_0,\tau)$,
which implies the desired result.
\end{proof}

\noindent{\bf Acknowlegement.} The author thanks Prof. Philippe Souplet for helpful suggestions during the preparation of this work. 
	 \smallskip
	 
\noindent{\bf Statements and Declarations.} The author states that there is no conflict of interest. 
This manuscript has no associated data.

\vspace{.5cm}
\begin{center}
	\textsc{Universit\'e Sorbonne Paris Nord (ex Paris 13), Institut Galil\'ee, Laboratoire Analyse G\'eom\'etrie et Applications, 99 Avenue Jean-Baptiste Cl\'ement 93430 Villetaneuse, France.}
	\\
	{\it Email address:} chabi@math.univ-paris13.fr
\end{center}

\end{document}